\documentclass{amsart}
\usepackage[all]{xy}
\usepackage{amsmath}
\usepackage{amsfonts}
\usepackage{amssymb}
\usepackage{amsthm}
\usepackage{amscd}
\usepackage{enumitem}
\usepackage{latexsym}
\usepackage{txfonts}
\usepackage{tikz}
\usepackage{graphicx}
\usepackage[all]{xy}
\usepackage{fullpage}
\newcommand{\bb}{\mathbb}
\newcommand{\rk}{\mrm{rk}_{\mathfrak{p}}}

\newcommand{\p}{\mathfrak{p}}
\newcommand{\la}{\lambda}
\newcommand{\K}{K^{\times}}
\newcommand{\ov}{\overline}
\newcommand{\R}{\mathcal{R}}
\newcommand{\B}{M^{\intercal}}
\newcommand{\Q}{Q^{\intercal}}
\newcommand{\C}{\mathcal{K}}
\newcommand{\SZ}{\mathcal{SZ}(G,A,B)}
\newcommand{\W}{\mathcal{W}(G,A,B,C,D)}
\newcommand{\Pa}{\mathcal{P}(q)}
\theoremstyle{plain}
\newtheorem{theorem}{Theorem}
\newtheorem{lemma}[theorem]{Lemma}
\newtheorem{corollary}[theorem]{Corollary}
\newtheorem{prop}[theorem]{Proposition}

\newtheorem*{theorem*}{Theorem}
\theoremstyle{remark} \newtheorem{remark}{Remark}
\theoremstyle{remark} 
\newcommand{\mrm}{\mathrm}

\begin{document}
\title{Difference families, skew Hadamard matrices, and Critical groups of doubly-regular tournaments.}
\date{}
\author{Venkata Raghu Tej Pantangi}
\email{pvrt1990@gmail.com}
\address{}
\begin{abstract}
In this paper we investigate the structure of the critical groups of doubly-regular tournaments (DRTs) associated with skew Hadamard difference families (SDFs) with one, two, or four blocks. In \cite{BR}, the existence of a skew Hadamard matrix of order $n+1$ was found to be  equivalent to the existence of a DRT on $n$ vertices. A well known construction of a skew Hadamard matrix order $n$ is by constructing skew Hadamard difference sets in abelian groups of order $n-1$. The Paley skew Hadamard matrix is an example of one such construction.  Szekeres \cite{SZ,SZ2} and Whiteman \cite{W} constructed skew Hadamard matrices from skew Hadamard difference families with two blocks. Wallis and Whiteman \cite{WW} constructed skew Hadamard matrices from skew Hadamard difference families with four blocks. In this paper we consider the critical groups of DRTs associated with skew Hadamard matrices constructed from skew Hadamard difference families with one, two or four blocks. We compute the critical groups of DRTs associated with skew Hadamard difference families with two or four blocks. We also compute the critical group of the Paley tournament and show that this tournament is inequivalent to the other DRTs we considered. Consequently we prove that the associated skew Hadamard matrices are not equivalent.   
\end{abstract}

\maketitle

\section{Introduction.}\label{intr}
A Hadamard matrix $H$ of order $n$ is an $n\times n$ matrix of $+1$'s and $-1$'s such that $HH^{\intercal}=nI$. It is well known that if $n$ is the order of a Hadamard matrix, then $n=1,2$ or $n\equiv 0 \pmod{4}$. It is conjectured that Hadamard matrices of order $n$ exist for all $n \equiv 0 \pmod{4}$. The smallest $n$ for which there is no known Hadamard matrix is $n=668$ (c.f. \cite{KH}). In this paper we deal with skew Hadamard matrices. A Hadamard matrix $H$ is said to be skew if $H+H^{\intercal}=2I$.

Two Hadamard matrices are considered equivalent if one can be obtained from the other by negating rows or columns, or by interchanging rows or columns. It is of interest to determine the equivalence of Hadamard matrices of the same order. Every Hadamard matrix of order $n$ is equivalent to a Hadamard matrix of the form 
$\left[\begin{smallmatrix}
1 & \mathbf{1}_{n-1}^{\intercal} \\
-\mathbf{1}_{n-1} & H_{0}
\end{smallmatrix} \right]$. All the Hadamard matrices we consider in this paper are assumed to be in this form.

In this paper, we are interested in the inequivalence of skew Hadamard matrices. If $H_{1}$ and $H_{2}$ are two Hadamard matrices of same order but with different Smith normal forms, then they are inequivalent. In \cite{MW}, it was found that the Smith normal form of any skew Hadamard matrix of order $4m$ is $\mrm{diag}[1,\underbrace{2,\ldots,\ 2}_{2m-1},\underbrace{2m,\ldots,\ 2m}_{2m-1}, 4m]$. So Smith normal form fails to distinguish inequvivalent skew Hadamard matrices of the same order. In this article we consider a different invariant associated with skew Hadamard matrices.

A tournament $T_{n}$ of order $n$ is a directed graph obtained by assigning directions to every edge of a complete graph on $n$ vertices. Given vertices $v,w$ of $T_{n}$, by $d(v)$ we denote the outdegree of $v$ and by $d(v,w)$ we denote the number of vertices dominated by $v$ and $w$. A doubly-regular tournamnet (DRT) with parameters $(n,k, \lambda)$ is a tournament of order $n$ such that for every pair of distinct vertices $v,\ w$, we have $d(v)=k$ and $d(v,w)=\lambda$. It is easy to see that $n=4\lambda +3$ and $k=2\lambda+1$. Theorem 2 of \cite{BR} shows that a skew Hadamard matrix of order $n+1$ exists if and only if there is a DRT on $n$ vertices. Given a Hadamard $H$ matrix of order $4\la+4$, the matrix $M$ which is obtained by deleting the first column and row of $\dfrac{1}{2}(J-H)$ is the adjacency matrix of a DRT with parameters $(4\la+3, 2\la+1,\la)$. Here $J$ is the matrix of all ones. Now if $\tilde{M}$ is the adjacency matrix of a DRT with parameters $(4\la+3, 2\la+1,\la)$, then $\tilde{H}=\left[\begin{smallmatrix}
1 & \mathbf{1}^{\intercal}_{4\la+3} \\
-\mathbf{1}_{4\la+3} & J-2\tilde{M}
\end{smallmatrix} 
\right]$ is a skew Hadamard matrix.
Any invariant of the DRT graph associated with a skew Hadamard matrix $H$ is an equivalence preserving invariant of $H$. In this paper, we look at the critical groups of DRTs.          

Let $\Gamma=(V,E)$ be a finite, connected, loopless, possibly directed graph on vertex set $V$ with edge set $E \subset V \times V$. We say that a vertex $v$ dominates a vertex $w$ if $(v,w) \in E$. By $\Delta_{v}$ we denote the number of vertices dominated by $v$.
By $\bb{Z}^{V}$ we denote the free abelian group with $V$ as a basis set. Then the adjacency map $\nu_{M}: \bb{Z}^{V} \to \bb{Z}^{V} $ that maps $v \in V$ to the formal sum of vertices dominated by $v$, encodes adjacency of the graph. The map $\nu_{Q}: \bb{Z}^{V} \to \bb{Z}^{V} $ that maps $v \in V$ to $\Delta_{v}v-\nu_{M}(v)$ is called the Laplacian map. The \emph{critical} group $\C$ of $\Gamma$ is the finite part of the cokernal of $\nu_{Q}$. The critical group is an invariant of the graph. Now let $\beta$ be an ordered basis of $\bb{Z}^{V}$ that is obtained by fixing an order on $V$. The adjacency matrix $M$ of $\Gamma$ is the matrix representation of $\nu_{M}$ with respect to $\beta$. We define the Laplacian matrix $Q$ to be the matrix representation of $\nu_{Q}$ with respect to $\beta$. Let $\Delta$ be the diagonal matrix whose $v$th diagonal entry is $\Delta_{v}$. Then we have $Q=\Delta-M$. If $Q_{v}$ is the matrix obtained by deleting the $v$th row and $v$th column of $Q$, then the Matrix-tree theorem (eg. \cite[5.64 and 5.68]{EC2}) states that $|det(Q_{v})|$ is the number of oriented trees in $\Gamma$ with root $v$. If $\Gamma$ is a directed Eulerian graph, $det(Q_{v})$ is independent of the vertex $v$ and $|\C|=|det(Q_{v})|$. If $\Gamma$ is undirected, $|\C|=|det(Q_{v})|$ is the number of spanning trees of $\Gamma$. For a nice survey on critical groups of graphs, we refer to \cite[\S 3]{sta}.
Some papers with computations of critical groups of families of graphs include \cite{Vince}, \cite{Dlor}, \cite{Lor1}, \cite{CSX}, \cite{Bai}, \cite{JNR}, \cite{DS}, \cite{SDB}, \cite{PS}, and \cite{P}. In \cite{Lor1}, Lorenzini examined the proportion of graphs with cyclic critical groups among graphs with critical groups of particular order.

One effective way of constructing skew Hadamard matrices/DRTs is by using skew difference families. Let $(G,+)$ be an additive finite abelian group of order $n$. A \emph{skew difference family} (SDF) on $l$ blocks with parameters $(n, k, \la)$ is a family $\{B_{i}| 1\leq i \leq l\}$ of $k$-subsets such that for all $1\leq i\leq l$ and $g\in G\setminus \{0_{G}\}$, we have (i) $|\{(x,y)\ \in \bigcup\limits_{i=1}^{l} B_{i}\times B_{i}\ |\ g=x-y\}|=\la-1$, (ii) $B_{i} \cap -B_{i} = \emptyset$, and (iii) $B_{i} \cup -B_{i}= G\setminus \{0_{G}\}$. An SDF with one block in $G$ is called a skew Hadamard difference set. 

We will now describe a few SDFs found in literature. The earliest construction is that of the Paley difference set by Paley \cite{Pal}. It was conjectured that Paley difference set was the only(upto equvivalence) SDF with one block. Ding and Yuan \cite{DY} disproved the conjecture by constructing other SDFs with one block. Szekeres \cite{SZ,SZ2}, Whiteman \cite{W} found an SDF with two blocks in $(\bb{F}_{q},+)$, where either $q\equiv 5 \pmod{8}$; or $q=p^{e}$ with $p \equiv 5 \pmod{8}$ a prime and $e\equiv 2\pmod{4}$. Wallis and Whiteman \cite{WW} constructed an SDF with four blocks in $(\bb{F}_{q},+)$, where $q \equiv 9 \pmod{16}$. Momihara and Xiang \cite{QM} generalised the constructions by Szekres, Wallis and Whiteman to obtain the following result. 

\begin{prop}\cite[Theorem 1.5]{QM}\label{MX}
Let $u\geq 2$ be an integer and $q$ be a prime power such that $q \equiv 2^u+1 \pmod {2^{u+1}}$. Then for any positive integer $e$, there exists a skew Hadamard difference family with $2^{u-1}$
blocks in $(\bb{F}_{q^{e}}, +)$.
\end{prop}

Szekeres \cite{SZ} also proved the following result.

\begin{prop}\cite[Theorem 3]{SZ}\label{SZ}
Let $q$ be a prime power such that $q \equiv 3 \pmod {4}$. Then, there exists a skew Hadamard difference family with $2$
blocks in $(\bb{Z}/n\bb{Z} , +)$, where $n=\frac{q-1}{2}$.
\end{prop}
      
In this paper, we compute the critical groups of the three families of DRTs described below. Given $X\subset G$, by $\delta_{X}$ we denote the characteristic function of $X$ in $G$.\\ 
(i) Let $(G,+)$ be an additive abelian group of order $2\la+1$ and $(A,B)$ be an SDF with two blocks in $G$, with parameters $(2\la+1, \la, \la-1)$.
Then $\SZ$ is the graph with vertex set $V=\{v_{0}\} \cup \{a_{g}| \ g \in G\} \cup \{b_{g}|\ g \in G\}$, whose adjacency map $\nu_{M}:\bb{Z}^{V}\to \bb{Z}^{V}$ satisfies
\begin{align}\label{sza}
\begin{split}
\nu_{M}(v_{0}) &= \sum\limits_{x\in G} a_{x}\\
\nu_{M}(a_{g}) &= \sum\limits_{z\in G} \delta_{A}(z) a_{g+z} + \sum\limits_{z\in G} \delta_{B \cup\{0_{G}\}}(z) b_{g+z}\\
\nu_{M}(b_{g}) &= v_{0}+\sum\limits_{z\in G} \delta_{-A}(z) b_{g+z} + \sum\limits_{z\in G} \delta_{B \cup\{0_{G}\}}(z) a_{g+z}\end{split},
\end{align} 
for all $g\in G$.
Theorem 2 of \cite{SZ} shows that $\SZ$ is a DRT with parameters $(4\la+3,2\la+1,\la)$. Setting $u=2$ in Proposition \ref{MX} provides us with a famiy of SDFs with two blocks. Proposition \ref{SZ} provides another such family. Theorem \ref{szc} describes the critical group of $\SZ$. 
We utilize the natural action of group $G$ on the vertex set of $\SZ$ to compute the Smith normal form of its Laplacian.

(ii)Let $(G,+)$ be an additive abelian group of order $2\la+1$ and $(A,B,C,D)$ be an SDF with four blocks in $G$, with parameters $(2\la+1, \la, \la-1)$. 

Then by $\W$ we denote a graph with vertex set $V=\{v_{1},v_{2},v_{3}\} \bigcup\limits_{\mu=a,b,c,d} V_{\mu}$. Here $V_{\mu}=\{\mu_{g}| \ g \in G\}$. We require the adjacency map $\nu_{M}:\bb{Z}^{V} \to \bb{Z}^{V}$ of $\W$ to satisfy
\begin{align}\label{wwa}
\begin{split}
\nu_{M}(v_{1})&= \sum\limits_{x\in G} a_{x}+ \sum\limits_{x\in G} c_{x} \\
\nu_{M}(v_{2}) &= \sum\limits_{x\in G} a_{x}+ \sum\limits_{x\in G} d_{x}\\
\nu_{M}(v_{3}) &= \sum\limits_{x\in G} a_{x}+ \sum\limits_{x\in G} b_{x}\\
\nu_{M}(a_{g}) &= \sum\limits_{z\in G} \delta_{A}(z) a_{g+z} + \sum\limits_{z\in G} \delta_{B \cup\{0_{G}\}}(z) b_{g+z} + \sum\limits_{z\in G} \delta_{-C \cup\{0_{G}\}}(z) c_{z-g} + \sum\limits_{z\in G} \delta_{D \cup\{0_{G}\}}(z) d_{(z+g)}\\
\nu_{M}(b_{g}) &= v_{1}+v_{2}+\sum\limits_{z\in G} \delta_{B}(z) a_{g+z} + \sum\limits_{z\in G} \delta_{-A}(z) b_{g+z} + \sum\limits_{z\in G} \delta_{-D}(z) c_{(g+z)} + \sum\limits_{z\in G} \delta_{C \cup\{0_{G}\}}(z) d_{z-g}\\
\nu_{M}(c_{g}) &= v_{2}+v_{3}+\sum\limits_{z\in G} \delta_{C}(z) a_{z-g} + \sum\limits_{z\in G} \delta_{-D \cup \{0_{G}\}}(z) b_{g+z} + \sum\limits_{z\in G} \delta_{A}(z) c_{(g+z)} + \sum\limits_{z\in G} \delta_{B}(z) d_{(g+z)}\\
\nu_{M}(d_{g}) &= v_{1}+v_{3}+\sum\limits_{z\in G} \delta_{D}(z) a_{g+z} + \sum\limits_{z\in G} \delta_{C}(z) b_{z-g} + \sum\limits_{z\in G} \delta_{B \cup \{0_{G}\}}(z) c_{(g+z)} + \sum\limits_{z\in G} \delta_{-A}(z) d_{(g+z)} 
\end{split},
\end{align} 
for all $g \in G$. 

Let $(g_{1},g_{2}, \ldots, g_{2\la+1})$ be an ordering on $G$. Consider the ordered basis
$$\beta= (v_{1},v_{2},v_{3}, a_{g_{1}},\ldots, a_{g_{2\la+1}},\ b_{g_{1}},\ldots, b_{g_{2\la+1}},\ c_{g_{1}},\ldots, c_{g_{2\la+1}},\ d_{g_{1}},\ldots, d_{g_{2\la+1}}).$$ Let $M$ be the matrix representation of $nu_{M}$ with respect to $\beta$.

Theorem 12 of \cite{WW} states that $\left[\begin{smallmatrix}
1 & \mathbf{1}^{\intercal}_{8\la+3} \\
-\mathbf{1}_{8\la+3} & J-2M
\end{smallmatrix} 
\right]$ is a skew Hadamard matrix. Using this we see that $\W$ is a DRT with parameters $(8\la+7,4\la+3,2\la+1)$. Setting $u=3$ in Proposition \ref{MX} provides us with a famiy of SDFs with four  blocks. Theorem \ref{wwc} describes the critical group of $\W$.
We utilize the natural action of group $G$ on the vertex set of $\W$ to compute the Smith normal form of its Laplacian.
   
(iii) 
The third family we consider is the family of Paley tournaments. Let $A$ be a skew Hadamard difference set in an abelian group $G$ of order $4\la+3$. By $DRT(G,A)$ we denote the graph with vertex set $\{[g]|\ g\in G\}$ and arc set $\{([g],[h])|\ h-g\in A\}$. The adjacency map $\nu_{M}:\bb{Z}^{G} \to \bb{Z}^{G}$ satisfies $\nu_{M}([g])= \sum\limits_{z\in G} \delta_{A}(z)[g+z]$.  
Let $p^{t}$ be a power of a prime $p$ with $q \equiv 3 \pmod{4}$ and let $\bb{F}_{q}$ be the finite field of order $q$. Let $H$ be the set of non-zero squares in $\bb{F}_{q}$. It is well known that $H$ is a skew Hadamard difference set in the additive group $(\bb{F}_{q},+)$ of the field. The Paley tournament graph $\mathcal{P}(q)$ is $DRT(G,H)$, that is, it is the Cayley graph on $(\bb{F}_{q}, +)$ with ``connection" set being the multiplicative subgroup of squares in $\bb{F}_{q}$. Theorem \ref{pc} describes the critical group of $\Pa$. This was essentially computed in \cite{CSX}, in which the authors describe the critical group of the Paley graph. This computation involves some Jacobi sums involving the quadratic character $\psi$. The only difference between our computation here and that in \cite{CSX} is that $\psi(-1)=-1$ in our case.

\section{Main results.}
Let $\C$ be the critical group of a DRT with parameters $(4\la+3,2\la+1,\la)$, by $\C_{1}$  we denote the subgroup of order $(\la+1)^{2\la+1}$. Let $\C_{2}$ be the subgroup of $\C$ of order $(4\la+3)^{2\la}$. We observe that $\C =\C_{1} \oplus\C_{2}$. In \S\ref{easy} we show that $\C_{1}$ depends only on the parameter $\la$.
\begin{theorem}\label{drtc}
Let $\la$ be a positive integer and let $\C$ denote the critical group of a DRT with parameters $(4\la+3, 2\la+1, \la)$. Then $\C= \left(\bb{Z}/ (\la+1)\bb{Z}\right)^{2\la+1} \oplus \C_{2}$, where $\C_{2}$ is a subgroup of order $(4\la+3)^{2\la}$. 
\end{theorem}

The result below describes the critical group of $\SZ$. We prove this in \S\ref{sc}.
\begin{theorem}\label{szc}
Let $\la$ be a positive integer and let $(A,B)$ be an SDF in an additive abelian $G$ with $|G|=2\la+1$. Let $Q$ denote the Laplacian matrix  of $\SZ$ and by $\C$ we denote its critical group. Then $\C= \left(\bb{Z}/(\la+1)\bb{Z}\right)^{2\la+1} \oplus \left(\bb{Z}/(4\la+3)\bb{Z}\right)^{2\la}$. Let $\p$ be a prime and let $\rk(Q)$ denote the $\p$-rank of $Q$. If $\p \mid \la+1$, then $\rk(Q)=2\la+1$; and if $\p \mid 4\la+3$, then $\rk(Q)=2\la+2$.  
\end{theorem}
The result below describes the critical group of $\W$. We prove this in \S\ref{wc}.
\begin{theorem}\label{wwc}
Let $\la$ be a positive integer and let $(A,B,C,D)$ be an SDF in an additive abelian $G$ with $|G|=2\la+1$. Let $Q$ denote the Laplacian matrix  of $\W$ and by $\C$ we denote its critical group. Then $\C= \left(\bb{Z}/(2\la+2)\bb{Z}\right)^{4\la+3} \oplus \left(\bb{Z}/(8\la+7)\bb{Z}\right)^{4\la+2}$. Let $\p$ be a prime and let $\rk(Q)$ denote the $\p$-rank of $Q$. If $\p \mid \la+1$, then $\rk(Q)=4\la+3$; and if $\p \mid 4\la+3$, then $\rk(Q)=4\la+4$.  
\end{theorem}
The following describes the critical group of $\Pa$. We prove this in \S\ref{pc}.
\begin{theorem}\label{ptc}
Let $p$ be a prime and $t$ be a positive integer such that $q:=p^{t}\equiv 3 \pmod{4}$. Let $Q$ denote the Laplacian matrix  of $\Pa$ and by $\C$ we denote its critical group. Then the $p$-rank of $Q$ is $\left(\dfrac{p+1}{2}\right)^{t}$ and 
$\C=\left(\bb{Z}/\mu\bb{Z}\right)^{2\mu}\bigoplus\limits_{i=1}^{t} \left(\bb{Z}/p^i \bb{Z}\right)^{e_{i}}$, where
\begin{enumerate}
\item $\mu=\dfrac{q-1}{4}$;
\item $e_{t}= \left(\dfrac{p+1}{2}\right)^{t}-2$;
\item and for $1\leq i <t$, $$e_{i}= \sum\limits_{j=0}^{min\{i, t-i\}} \dfrac{t}{t-j} {t-j \choose j} {t-2j \choose i-j} (-p)^{j} \left(\dfrac{p+1}{2}\right)^{t-2j}.$$
\end{enumerate} 
\end{theorem}
\begin{remark}
Let $q$ be a prime power satisfying $q\equiv 3 \pmod 4$. Proposition \ref{SZ} provides us with an SDF with two blocks in $G:=\bb{Z}/n\bb{Z}$, where $n=\frac{q-1}{2}$. Let $(A,B)$ be an SDF in $G$. Both $\Pa$ and $\SZ$ are DRT's with parameters $({q}, ({q}-1)/2, ({q}-3)/4)$. Theorems \ref{szc} and \ref{ptc} show that these graphs have non isomorphic critical groups. Therefore these graphs are not isomorphic and thus the associated Hadamard matrices are not equivalent.  
\end{remark}
\begin{remark}
Let $\tilde{q}$ be a prime power such that $\tilde{q}\equiv 5 \pmod{8}$. Proposition \ref{MX}  guarantees the existence of an SDF with two blocks in $(\bb{F}_{q},+)$.
Let $(A,B)$ be an SDF in $(\bb{F}_{\tilde{q}},+)$.
Let's also assume that
$q=2\tilde{q}+1$ is also a power of a prime. Theorems \ref{szc} and \ref{ptc} show that that $\mathcal{SZ}(\bb{F}_{\tilde{q}},A,B)$ and $\mathcal{P}(q)$ are not isomorphic and thus the associated Hadamard matrices are not equivalent.   
\end{remark}
\begin{remark}
Let $\tilde{q}$ be a prime power such that $\tilde{q}\equiv 9 \pmod{16}$. Proposition \ref{MX}  guarantees the existence of an SDF with four blocks in $(\bb{F}_{q},+)$.
Let $(A,B,C,D)$ be an SDF in $(\bb{F}_{\tilde{q}},+)$.
 Let's also assume that
$q=4\tilde{q}+3$ is also a power of a prime. Theorems \ref{wwc} and \ref{ptc} show that that $\mathcal{W}(\bb{F}_{\tilde{q}},A,B,C,D)$ and $\mathcal{P}(q)$ are not isomorphic and thus the associated Hadamard matrices are not equivalent.   
\end{remark}

\begin{remark}
We found that the critical groups of $\SZ$ and $\W$ depend only on the order of $G$. However this is not the case for DRTs constructed from skew Hadamard difference sets. Let $q \equiv 3 \pmod{4}$ be a prime power. To construct $\Pa$, the set $H$ of quadratic residues in $(\bb{F}_{q},+)$ was used. Another example of skew Hadamard difference set is the set $DY(1)=\{x^{10}-x^6-x^{2}| \ x \in \bb{F}_{3^{n}}^{\times}\}$ in the additive group $(\bb{F}_{3^{n}},+)$, with $n$ odd. This was constructed by Ding and Yuan \cite{DY}. 
By $DRT(3^{n}, DY(1))$, we denote the DRT with vertex set $\{[x]| \ x \in \bb{F}_{3^{n}}\}$ 
and arc set $\{([x],[y])|\ y-x \in DY(1)\}$. With the help of a computer, we can find that the SNFs of the Laplacians of $DRT(3^{5}, DY(1))$ and $\mathcal{P}(243)$ are different. It was conjectured in \cite{DWX} that there are at least five inequvivalent difference sets in $(\bb{F}_{3^{n}},+)$ for all odd $n>3$. 
\end{remark}
\section{Preliminaries}\label{def}
\subsection{Smith Normal Forms.}
Let $\mathfrak{R}$ be a Principal Ideal Domain and $Z:\mathfrak{R}^{m} \to \mathfrak{R}^{n}$ be a linear transformation. 
By the structure theorem for finitely generated modules over PIDs, we have $\{s_{i}(Z)\}_{i=1}^{r} \subset \mathfrak{R} \setminus \{0\}$ such that $s_{i}(Z) \mid s_{i+1}(Z)$ and 
$$\mathrm{coker}(Z)\cong \mathfrak{R}^{n-r}\oplus \bigoplus\limits_{i=1}^{r} \mathfrak{R} /s_{i}(Z)\mathfrak{R}.$$  
Let $[Z]$ denote the matrix representation of $Z$ with respect to the standard bases. Then the above equation tells us that we can find $P \in \mathrm{GL}_{n}(\mathfrak{R})$, and $Q \in \mathrm{GL}_{m}(\mathfrak{R})$ such that 
$$P[Z]Q=\left[
\begin{array}{c | c}
Y & O_{(r \times n-r)} \\
\hline
O_{(m-r \times r)} & O_{(n-r \times n-r)}
\end{array}
\right],$$ 
where $Y=\mrm{diag}(s_{1}(Z), \ldots ,\ s_{r}(Z))$. The diagonal form $P[Z]Q$ is called the Smith normal form (SNF) of $Z$. Its uniqueness (up to multiplication of $s_{i}(Z)$'s by units) is also guaranteed by the aforementioned structure theorem. By invariant factors (elementary divisors) of $Z$, we mean the invariant factors (respectively elementary divisors) of the module $\mathrm{coker}(Z)$. In this section, we collect some useful results about Smith normal forms.

The following is a well known result (for eg. see Theorem $2.4$ of \cite{sta}) that gives a description of the Smith normal form in terms of minor determinants.

\begin{lemma}\label{minor}
Let $Z$, $[Z]$, and $\{s_{i}(Z)\}_{1\leq i \leq r}$ be as described above.
Given $1 \leq i \leq r$, let $d_{i}(Z)$ be the GCD of all $i\times i$ minor determinants of $[Z]$, and let $d_{0}(Z)=1$. We then have $s_{i}(Z)=d_{i}([Z])/d_{i-1}([Z])$. 
\end{lemma}

The following result which is Theorem 1 of \cite{MN71} gives a relation between SNF of the product of two matrices and the SNFs of the individual matrices.

\begin{lemma}\label{snfprod}
Let $\mathfrak{R}$ be a principal ideal domain. Given $M \in M_{n}(\mathfrak{R})$ and $1\leq k\leq n$, by $s_{k}(M)$ we denote the $k$th invariant factor of $M$. If $A,\ B \in M_{n}(\mathfrak{R})$, then for $1\leq k \leq n$ we have $s_{k}(A)\mid s_{k}(AB)$ and $s_{k}(B) \mid s_{k}(AB)$.
\end{lemma}
Consider a prime $\p \in \mathfrak{R}$ and a square matrix $N$ with entries in $\mathfrak{R}$, whose SNF over $\mathfrak{R}$ is 

$\mrm{diag}(s_{1}(N),\ldots, s_{i}(N), \ldots s_{n}(N))$.

Let $\mathcal{S}_{\p}$ be any unramified extension of the local ring $\mathfrak{R}_{\p}$. 
If $\mrm{diag}(\p^{j_{1}},\ldots, \p^{j_{i}},\ldots, \p^{j_{n}})$ is the SNF of $N$ considered as a matrix over $\mathcal{S}_{\p}$, then $\p^{j_{i}} || s_{i}(N)$. So  while finding Smith normal forms, we can focus on one prime at a time.
\subsection{Properties of DRTs.}
Let $\la$ be a positive integer. By $M$ be we denote the adjacency matrix of a DRT with parameters $(4\lambda +3, 2\la+1, \la)$, then $Q:=(2\la +1)I-M$ is its Laplacian matrix.
Using the definition of DRTs, we can easily deduce that $M+\B=J-I$ and $M \B =(\la +1)I+\la J$. Thus we have
\begin{equation}\label{ae}
 Q+Q^{\intercal}=(4\la+3)I-J
 \end{equation} and 
\begin{equation}\label{me} 
 QQ^{\intercal}=(4\la +3)(\la +1)I-(\la+1)J.
 \end{equation} 

The following is a well know result about adjacency matrices of DRTs.
\begin{lemma}\label{linpar}
Let $\la$ be a positive integer and let $\Gamma$ be a DRT with parameters $(4\la+3,2\la+1, \la)$.
Let $M$ and $Q$ be the adjaceny matrix and Laplacian matrix respectively, of $\Gamma$. By $\C$ we denote the critical group of $\Gamma$. Then
\begin{enumerate}[label=\roman*]
\item $(x-k)(x^{2}+x+\la+1)^{2\la+1}$ is the characteristic polynomial of $M$;
\item the eigenvalues of $Q$ are $0$, $\dfrac{4\la+3-\left(\sqrt{4\la+3}\right)i}{2}$, and $\dfrac{4\la+3+\left(\sqrt{4\la+3}\right)i}{2}$, with multiplicities $1$, $k$, and $k$ respectively;
\item and $|\C|=(4\la+3)^{2\la}(\la+1)^{2\la+1}$.
\end{enumerate} 
\end{lemma} 
\begin{proof}
Using $Q= (2\la+1)I-M$, we see that (i) implies (ii). Matrix tree theorem shows that (ii) implies (iii).

Using $M+\B=J-I$ and $M \B =(\la +1)I+\la J$, we have $\mrm{det}(xI-M) \mrm{det}(xI-M^{\intercal})= \mrm{det}((x-\la)J+(x^{2}+x+\la+1)I)$. Observing that $\mrm{det}(aI+bJ)=(a+(4\la+3)b)a^{4\la+2}$ and that $\mrm{det}(xI-M)=\mrm{det}(xI-M^{\intercal})$, we arrive at (i). 
\end{proof}
\subsection{Permutation action and characters.}
We will now collect some useful results from character theory. Each of $\Pa$, $\SZ$, $\W$ is constructed using a finite abelian group $(G,+)$. We use the natural action of $G$ on the vertex set to compute the critical groups. These actions are closely related to the regular action of $G$ on itself.

We define the action of $G$ on $Y=\{y_{g}|\ g\in G\}$ by $h.y_{g}=y_{g+h}$. This is the regular action of $G$. Let $\p \nmid |G|$ be a prime and let $\mathfrak{S}$ be an extension of $\bb{Q}_{\p}$ containing the $|G|$-th roots of unity. By $R$ we denote the ring of integers of $\mathfrak{S}$, and by $R^{Y}$ we denote the free $R$-module generated by $Y$ as a basis set. Let $Irr(G)$ be the group of $R$-valued characters of $G$.

It is well known from representation theory that the $RG$-permutation module $R^Y$ decomposes into direct sum of non-isomorphic $RG$-modules of $R$-rank $1$, affording characters $\chi \in Irr(G)$. A basis element for the module affording $\chi$ is $e_{(y,\chi)}=\sum\limits_{g\in G} \chi(-g) y_{g}$. The following Lemma is useful in our computations. 
\begin{lemma}\label{use}
Let $X\subset G$, and let $\delta_{X}:G \to R$ be the characteristic function of $X$ in $G$. Let $\chi(X):= \sum\limits_{z\in X}\chi(z)$. Then we have
\begin{enumerate}
\item $\sum\limits_{g\in G} \chi(-g)\sum\limits_{z\in G}\delta_{X}(z)y_{g+z}=\chi(X)e_{(y,
\chi)}$ and
\item $\sum\limits_{g\in G} \chi(-g)\sum\limits_{z\in G}\delta_{X}(z)y_{z-g}=\ov{\chi(X)}e_{(y,
\chi^{-1})}$.
\end{enumerate}
\end{lemma}    
\begin{proof}
Using $\chi(-g)=\chi(-z)\chi(z-g)=\chi^{-1}(z)\chi^{-1}(g-z)$, we see that  $$\sum\limits_{z\in G}\chi^{-1}(z) \delta_{X}(z) \sum\limits_{g\in G} 
\chi^{-1}(g-z)y_{z-g}= (\chi^{-1}(X))e_{(y,\chi^{-1})}.$$ We may conclude (2) by using $\ov{\chi(X)}=\chi^{-1}(X)$. Proof of (1) follows via similar rearrangements.
\end{proof}
\section{Description of $\C_{1}$.}\label{easy} 
In this section we prove Theorem \ref{drtc}. Let $Q$ be the Laplacian matrix of a DRT with parameters $(4\la+3, 2\la+1, \la)$. Let $\C$ denote its critical group, then from Lemma \ref{linpar}, we have $|\C|=(4\la+3)^{2\la}(\la+1)^{2\la+1}$. As $4\la+3$ and $\la+1$ are coprime,there are subgroups $\C_{1}$, $\C_{2}$ of $\C$ such that $|\C_{1}|=(\la+1)^{2\la+1}$, $|\C_{2}|=(4\la+3)^{2\la}$, and $\C=\C_{1} \oplus \C_{2}$. Theorem \ref{drtc} describes the structure of $\C_{1}$ for all DRTs.

We extend the arguments used in \cite{MW} to determine the structure of $\C_{1}$. Let $\p\mid \la+1$ be a prime integer. Then the Smith normal form of $Q$ over the $\p$-adic numbers $\bb{Z}_{\p}$ gives us the $\p$-part of $\C_{1}$. Let $s_{1}(Q), \ldots, s_{4\la+3}(Q)$ be the invariant factors of $Q$ considered as a matrix over $\bb{Z}_{\p}$  

From \eqref{me} we see that $Q\Q\equiv O \pmod{\p}$. Therefore we have $\mrm{rk}_{\mathfrak{p}}(Q)\leq 4\la+3-\mrm{rk}_{\mathfrak{p}}(\Q)$ and thus $\mrm{rk}_{\mathfrak{p}}(Q) \leq 2\la+1$. Using \eqref{ae}, we see that $Q + \Q \equiv -I-J \pmod{\mathfrak{p}}$. So,
\begin{align*}
4\la+2 & =\mrm{rk}_{\mathfrak{p}}(I+J)\\
&=\mrm{rk}_{\mathfrak{p}}(Q+\Q)\\
& \leq 2\mrm{rk}_{\mathfrak{p}}(Q),
\end{align*}      
and thus $\rk(Q)= 2\la +1$. So $s_{i}(Q)$ is a unit in $\bb{Z}_{\p}$ for $1\leq i \leq 2\la+1$. 

Since the SNF of $\left((4\la +3)(\la +1)I-(\la+1)J\right)$ is $\mrm{diag}(\la+1,\ (4\la +3)(\la +1), \ldots,\ (4\la +3)(\la +1),\ 0)$, equation \eqref{me} and Lemma \ref{snfprod} can be used to conclude that (i) $s_{1}(Q)\mid s_{1}(Q\Q)=\la+1$; (ii) $s_{4\la+3}(Q)\mid s_{4\la+3}(Q\Q)=0$; (iii) and $s_{i}(Q)\mid (4\la+3)(\la+1)$ for $1< i <4\la+3$. We recall that $v_{\p}\left(\prod\limits_{i=1}^{4\la+2}s_{i}(Q) \right)=v_{\p}\left((4\la+3)^{2\la}(\la+1)^{2\la+1}\right)$. Since for $1\leq i \leq 2\la+1$, $s_{i}(Q)$ is a unit in $\bb{Z}_{p}$, we have $v_{\p}\left(\prod\limits_{i=2\la+2}^{4\la+2}s_{i}(Q) \right)=v_{\p}\left((4\la+3)^{2\la}(\la+1)^{2\la+1}\right)$. As $s_{i}(Q) \mid \la+1$, we can now conclude that $s_{i}(Q)=\la+1$ for $2\la+2\leq i \leq 4\la+2$. It now follows that $\C_{1}= \left(\bb{Z}/(\la+1)\bb{Z}\right)^{2\la+1}$.

\section{Crtical group of $\SZ$.}\label{sc}
Let us turn our attention to DRTs of the form $\SZ$. Let $(A,B)$ be an SDF in an abelian group $(G,+)$ of order $2\la+1$. By $\SZ$ we denote the graph with vertex set $V=\{v_{0}\} \cup \{a_{g}| \ g \in G\} \cup \{b_{g}|\ g \in G\}$ and whose adjacency operator $\nu_{M}$ satisfies \eqref{sza}. 

We recall that $\SZ$ is a DRT with parameters $(4\la+3,2\la+1,\la)$. Let $Q$ be the Laplacian matrix of $\SZ$ and $\C$ be its critical group. By Theorem \ref{drtc}, we have $\C=\left(\bb{Z}/ (\la+1)\bb{Z}\right)^{2\la+1} \oplus \C_{2}$, where $\C_{2}$ is the subgroup of order $(4\la+3)^{2\la}$. Let $\p \mid 4\la+3$ be a prime. 
Determining the SNF of $Q$ over an unramified extension of $\bb{Z}_{\p}$ will give us the $\p$-part of $\C_{2}$.

Let $t \in \bb{N}$ such that $|G|\mid (\p^{t}-1)$, by $\theta$ we denote a primitive $(\p^{t}-1)$st root of unity in $\bb{Q}_{\p}$. We denote by $\R$, the ring of integers in $\bb{Q}(\theta)$. As $\p$ is unramified in $\R$, the $\p$-part of $\C$ can be found by determining the SNF of $Q$ over $\R$. By $\R^{V}$, we denote the free module over $\R$ with $V$ as a basis set. The matrix $Q$ defines a map $\nu_{Q}:\R^{V} \to \R^{V}$ with $\nu_{Q}(x)=(2\la+1)x-\nu_{M}(x)$.

We now consider the action of $(G,+)$ on $V$ that satisfies (i)$h.v_{0}=v_{0}$, (ii) $h.a_{g}=a_{g+h}$, and (iii)$h.b_{g}=b_{g+h}$ for all $g,h \in G$. This permutation action preserves adjacency. The action of $G$ on $V$ makes $\R^{V}$ a permutation module for $G$.  Given $\chi \in Irr(G)$, we define $N_{\chi}:=\{x \in \R^{V}| g.x=\chi(g)x\ \text{for all}\ g\in G\}$. In other words, $N_{\chi}$ is the direct sum of all irreducible submodules of $\R^{V}$ affording $\chi$. As $G$ preserves adjacency, $\nu_{Q}$ is an $\R G$ map. Now by applying Schur's Lemma, we have $\nu(N_{\chi})\subset N_{\chi}$. 
  
The action of $G$ decomposes $\R^{V}$ into $\R^{\{v_{0}\}}\oplus \R^{V_{a}} \oplus \R^{V_{b}}$, where $V_{\mu}=\{\mu_{g}|g \in G\}$ for $\mu=a,b$. Now $\R^{V_{\mu}}$ is a regular module for $G$, and $\R^{\{v_{0}\}}$ is a trivial module. Now $\R^{V_{\mu}}= \bigoplus\limits_{\chi \in Irr(G)} M_{(\chi,\mu)}$, where $M_{(\chi,\mu)}$ is the submodule affording $\chi$. A basis element for $M_{(\chi,\mu)}$ is $e_{(\mu,\chi)}=\sum\limits_{g\in G} \chi(-g)\mu_{g}$.

Let $\chi_{0}$ denote the trivial character of $G$. For $\chi \neq \chi_{0}$, we have $N_{\chi}=M_{(\chi,a)}+M_{(\chi,b)}$; and $N_{\chi_{0}}= Rv+M_{(\chi_{0},a)}+M_{(\chi_{0},b)}$. We now have $R^{V}= \bigoplus\limits_{\chi \in Irr(G)} N_{\chi}$, with $\nu_{Q}(N_{\chi}) \subset N_{\chi}$. We will now look at $\nu_{Q}|_{N_{\chi}}$.

Using Lemma \ref{use} and the relations in \eqref{sza} yields the following lemma. 
\begin{lemma}\label{imsz}
Let $\chi\in Irr(G)\setminus \{\chi_{0}\}$, then
\begin{enumerate}
\item $\nu_{Q}(e_{(a,\chi)})=(2\la+1-\chi(A)) e_{(a,\chi)} + (-1-\chi(B)) e_{(b,\chi)}$ and
\item $\nu_{Q}(e_{(b,\chi)})=(-\chi(B)) e_{(a,\chi)} + (2\la+1-\chi(-A)) e_{(b,\chi)}$.
\end{enumerate}
\end{lemma} 

For $\chi\neq \chi_{0}$, let $Q_{\chi}$ be the matrix representation of $\nu_{Q}|_{N_{\chi}}$ with respect to the ordered basis $(e_{(a,\chi)}, e_{(b,\chi)})$. Let $Q_{\chi_{0}}$ be the matrix representation of $\nu_{Q}|_{N_{\chi_{0}}}$ with respect to the ordered basis $(e_{(a,\chi_{0})}, e_{(b,\chi_{0})},v)$. So $Q$ is similar to the block diagonal matrix $\bigoplus\limits_{\chi\in Irr(G)}Q_{\chi}$.

 We see from Lemma \ref{imsz} that for $\chi \neq \chi_{0}$, we have $Tr(Q_{\chi})=4\la+2- \chi(A)-\chi(-A)$. Now as $\chi$ is not trivial, we have $0=\chi(G)$. Using $A \cup -A = G\setminus \{0_{G}\}$, we may conclude that $Tr(Q_{\chi})=4\la+3$. The eigenvalues of $Q_{\chi}$ are elements of the set $\left\lbrace 0, \dfrac{4\la+3-\left(\sqrt{4\la+3}\right)i}{2}, \dfrac{4\la+3+\left(\sqrt{4\la+3}\right)i}{2}\right\rbrace$ of eigenvalues of $Q$. As $Tr(Q_{\chi})=4\la+3$, the eigenvalues of $Q_{\chi}$ are $\dfrac{4\la+3-\left(\sqrt{4\la+3}\right)i}{2}$ and $\dfrac{4\la+3+\left(\sqrt{4\la+3}\right)i}{2}$ and $det(Q_{\chi})=(4\la+3)(\la+1)$.

We also observe from Lemma \ref{imsz} that the difference of the off-diagonal entries of $Q_{\chi}$ is $1$ and thus one of them is coprime to $\p$. Applying Lemma \ref{minor} and $det(Q_{\chi})=(4\la+3)(\la+1)$ we can conclude that $\mrm{diag}(1,4\la+3)$ is the SNF of $Q_{\chi}$ over $\R$. Similar computations can be used to show that $\mrm{diag}(1,1,0)$ is the SNF of $Q_{\chi_{0}}$ over $\R$. This proves Theorem \ref{szc}.

\section{Crtical group of $\W$.}\label{wc}
Given an SDF $(A,B,C,D)$ in an additive abelian group $(G,+)$ of order $2\la+1$, we recall that $\W$ is a graph with vertex set $V=\{v_{1},v_{2},v_{3}\} \bigcup\limits_{\mu=a,b,c,d} \{\mu_{g}| \ g \in G\}$, and whose adjacency operator $\nu_{M}$ is defined by \eqref{wwa}.

We recall that $\W$ is a DRT with parameters $(8\la+7,4\la+3,2\la+2)$. Let $Q$ be the Laplacian matrix of $\W$ and $\C$ be its critical group. By Theorem \ref{drtc}, we have $\C=\left(\bb{Z}/ (2\la+2)\bb{Z}\right)^{4\la+3} \oplus \C_{2}$, where $\C_{2}$ is the subgroup of order $(8\la+7)^{4\la+2}$. Let $\p \mid 8\la+7$ be a prime.
Determining the SNF of $Q$ over an unramified extension of $\bb{Z}_{\p}$ will give us the $\p$-part of $\C_{2}$.

Let $t \in \bb{N}$ such that $|G|\mid (\p^{t}-1)$, by $\theta$ we denote a primitive $(\p^{t}-1)$st root of unity in $\bb{Q}_{\p}$. We denote by $\R$, the ring of integers in $\bb{Q}(\theta)$. As $\p$ is unramified in $\R$, the $\p$-part of $\C$ can be found by determining the SNF of $Q$ over $\R$. By $\R^{V}$, we denote the free module over $\R$ with $V$ as a basis set. The matrix $Q$ defines a map $\nu_{Q}:\R^{V} \to \R^{V}$ with $\nu_{Q}(x)=(4\la+3)x-\nu_{M}(x)$. 

Unlike in the case of $\SZ$, the natural $G$ action on $\W$ does not preserve adjacency, but provides a useful integral basis for $\R^{V}$. We consider the action of $G$ on $V$ that satisfies (i)$h.v_{i}=v_{i}$, (ii) $h.\mu_{g}=\mu_{g+h}$ for all $(i,g,h, \mu) \in \{1,2,3\}\times G \times G \times\{a,b,c,d\}$.

The action of $G$ decomposes $\R^{V}$ into $\bigoplus\limits_{i=1}^{3}\R^{\{v_{i}\}}\oplus \bigoplus\limits_{\mu=a,b,c,d} \R^{V_{\mu}}$, where $V_{\mu}=\{\mu_{g}|g \in G\}$. Now $\R^{V_{\mu}}$ is a regular module of $G$, and $\R^{\{v_{i}\}}$'s are trivial modules. Now $\R^{V_{\mu}}= \bigoplus\limits_{\chi \in Irr(G)} M_{(\chi,\mu)}$, where $M_{(\chi,\mu)}$ is the submodule affording $\chi$. A basis element for $M_{(\chi,\mu)}$ is $e_{(\mu,\chi)}=\sum\limits_{g\in G} \chi(-g)\mu_{g}$. The following Lemma describes the images of $e_{(\mu,\chi)}$ under the action of $\nu_{Q}$. 

\begin{lemma}\label{imww}
For non-trivial $\chi \in Irr(G)$, we have
\begin{enumerate}
\item $\nu_{Q}(e_{(a,\chi)})=(4\la+3-\chi(A))e_{(a,\chi)}+(\ov{\chi(B)})e_{(b,\chi)}+ (\ov{\chi(C)})e_{(c,\chi^{-1})}+ (\ov{\chi(D)})e_{(d,\chi)}$
\item  $\nu_{Q}(e_{(b,\chi)})=-\chi(B)e_{(a,\chi)}+ (4\la+3-\ov{\chi(A)})e_{(b,\chi)}- (\ov{\chi(D)})e_{(c,\chi)}+ (\ov{\chi(C)})e_{(d,\chi^{-1})}$
\item $\nu_{Q}(e_{(c,\chi)})=\ov{-\chi(C)}e_{(a,\chi^{-1})}+ (\chi(D))e_{(b,\chi)}+ (4\la+3-\chi(A))e_{(c,\chi)}-(\chi(B))e_{(d,\chi)}$
\item $\nu_{Q}(e_{(d,\chi)})=-\chi(D)e_{(a,\chi)}- (\ov{\chi(C)})e_{(b,\chi^{-1})}+ (\ov{\chi(B)})e_{(c,\chi)}+ (4\la+3-\ov{\chi(A)})e_{(d,\chi)}$.
\end{enumerate}
\end{lemma}    

The result above follows by straightforward applications of the relations in \eqref{wwa} and Lemma \ref{use}. 

Let $\chi_{0}$ denote the trivial character of $G$.
For $\chi\neq \chi_{0} \in Irr(G)$, define $N_{\chi}$ to be the $\R^{V}$ submodule generated by $\{e_{(\mu,f)}|\ \mu=a,b,c,d\ \text{and}\ f=\chi,\chi^{-1}\}$. Let $N_{\chi_{0}}$ be the submodule generated by $\{v_{1},v_{2},v_{3},e_{(\mu,\chi_{0})}|\ \mu=a,b,c,d\}$. Lemma \ref{imww} shows that $\nu_{Q}(N_{\chi}) \subset N_{\chi}$ for $\chi \neq \chi_{0}$. Lemma \ref{imww} implies $\nu_{Q}(N_{\chi_{0}}) \subset N_{\chi_{0}}$. By $\nu_{\chi}$ we denote $\nu_{Q}|_{N_{\chi}}$.

For $\chi\neq \chi_{0}$, let $Q_{\chi}$ be the matrix representation of $\nu_{\chi}$ with respect to the ordered basis $(e_{(\mu,\chi)}|\ {\mu=a,b,c,d}) \cup (e_{(\mu,\chi^{-1})}|\ {\mu=a,b,c,d})$ (see \eqref{mat}). By Lemma \ref{imww} we have $Tr(\nu_{\chi})=16\la+12-2(\chi(A)+\chi(-A))=2(8\la+7)$. The eigenvalues of $Q_{\chi}$ are elements of the set $\left\lbrace 0, \dfrac{8\la+7-\left(\sqrt{8\la+7}\right)i}{2}, \dfrac{8\la+7+\left(\sqrt{8\la+7}\right)i}{2}\right\rbrace$ of eigenvalues of $Q$. As $Tr(Q_{\chi})=8\la+7$, the eigenvalues of $Q_{\chi}$ are $\dfrac{8\la+7-\left(\sqrt{8\la+7}\right)i}{2}$ and $\dfrac{8\la+7+\left(\sqrt{8\la+7}\right)i}{2}$ and that $det(Q_{\chi})=(8\la+7)^{4}(2\la+2)^{4}$. 

\begin{equation}\label{mat}
Q_{\chi}=  \left(\begin{smallmatrix}
4\la+3- \chi(A) & -\chi(B) & 0 & -\chi(D) & 0 & 0 & -\chi(C) & 0\\
\ov{\chi(B)} & 4\la+3- \ov{\chi(A)} & \chi(D) & 0 &0 & 0 &0 & -\chi(C)\\
0 & -\ov{\chi(D)} & 4\la+3- \chi(A) & \ov{\chi(B)} & \ov{\chi(C)} & 0 & 0 & 0\\
-\chi(D) & 0 & -\chi(B) & 4\la+3- \ov{\chi(A)} & 0 & \chi(C) & 0 & 0\\
0 & 0 & -\ov{\chi(C)} & 0 & 4\la+3- \ov{\chi(A)} & -\ov{\chi(B)} & 0 & -\ov{\chi(D)}\\
0 & 0 &0 & -\ov{\chi(C)} & \chi(B) & 4\la+3- \chi(A) & \ov{\chi(D)} & 0 \\
\chi(C) & 0 & 0 & 0 & 0 & -{\chi(D)} & 4\la+3- \ov{\chi(A)} & \chi(B)\\
0 & \ov{\chi(C)} & 0 & 0 & -\ov{\chi(D)} & 0 & -\ov{\chi(B)} & 4\la+3- \chi(A)
\end{smallmatrix}  \right). 
\end{equation}

Straight forward computations show that $Q_{\chi}\ov{Q_{\chi}^{\intercal}}=sI$, where $s=|\chi(A)|^{2}+|\chi(B)|^{2}+|\chi(C)|^{2}+|\chi(D)|^{2}$. As $det(Q_{\chi})=(8\la+7)^{4}(2\la+2)^{4}$, we have $s=(8\la+7)(2\la+2)$. 

Let $m_{1}$ be the minor of $Q_{\chi}$  associated to row indices $\{1,3,5,7\}$ and columns indices $\{1,3,5,7\}$. Let $m_{2}$ be the minor of $Q_{\chi}$ associated to row indices $\{1,3,5,7\}$ and columns indices $\{2,4,7,8\}$. Computations yield  $m_{1}=((4\la+3)^{2}+(4\la+3)+|\chi(A)|^{2}+|\chi(C)|^{2})^{2}$ and $m_{2}=(|\chi(B)|^{2}+|\chi(D)|^{2})^{2}$. We see that $\sqrt{m_{2}}+\sqrt{m_{1}}= (4\la+3)^{2}+(4\la+3)+(8\la+7)(2\la+2)=(4\la+3)(4\la+4)+(8\la+7)(2\la+2)$. As both $4\la+4$ and $4\la+3$ are coprime to $8\la+7$, we see that $\p$ does not divide $m_{1}$ and $m_{2}$ simultaneously. So there is at least one $4$-minor of $Q_{\chi}$ that is not divisible by $\p$. Applying Lemma \ref{minor} and $det(Q_{\chi})=(8\la+7)^{4}(2\la+2)^{4}$ we can conclude that the SNF of $Q_{\chi}$ over $\R$ is of the form $\mrm{diag}(1,1,1,1,e_{5},e_{6},e_{7},e_{8})$, where $e_{5}\ \mid e_{6}\ \mid e_{7}\ \mid e_{8}$ and $v_{\p}(e_{5}e_{6}e_{7}e_{8})=4v_{\p}(8\la+7)$. As $Q_{\chi}\ov{Q_{\chi}^{\intercal}}=(8\la+7)(2\la+2)I$, we can conclude that $v_{\p}(e_{i})=v_{\p}(8\la+7)$ for $i=5,6,7,8$. This concludes the proof of Theorem \ref{wwc}.    
\section{Critical group of $\Pa$.}\label{pc}
We now turn our attention to Paley tournament graph $\Pa$. The computation of critical group of $\Pa$ is essentially the same as that of the Paley graph done in \cite{CSX}. The proofs of results in this section are similar to those in \cite{CSX}.

Let $q=p^{t}$ be a power of a prime $p$ with $q \equiv 3 \pmod{4}$. Let $K$ be the group field with $q$ elements and let $H$ be the subgroup of squared in $\K$. We recall that the Paley tournament graph $\mathcal{P}(q)$ is the Cayley graph of $(K, +)$ with ``connection" set being $H$.  

$\Pa$ is a DRT with parameters $\left(q,k:=\dfrac{q-1}{2}, \la:=\dfrac{q-3}{4}\right)$. Let $Q$ be the Laplacian matrix of $\Pa$ and $\C$ be its critical group . By Theorem \ref{drtc}, we have $\C=\left(\bb{Z}/ (\mu)\bb{Z}\right)^{2\mu} \oplus \C_{2}$, where $\mu=\dfrac{q-1}{4}$ and $\C_{2}$ is the subgroup of order $q^{\frac{q-3}{2}}$. So we now need to determine the Sylow $p$-subgroup of $\C$. We do this by determining the SNF of $Q$ over an unramified extension of $\bb{Z}_{p}$.

Let $R$ be the ring of integers of the unique unramified extension of degree $t$ over $\bb{Q}_{p}$. Then the ideal $pR$ is a maximal ideal, and thus $K=R/pR \cong \bb{F}_{q}$. By $R^{K}$, we denote the free module over $R$ generated by $\{[x]|\ x\in K\}$. The matrix $Q$ defines a map $\nu_{Q}:R^{K} \to R^{K}$ that satisfies $\nu_{Q}([x])=k[x]- \sum\limits_{z\in S}[x+z]$. In other words $Q$ is the matrix representation of $\nu_{Q}$ with respect to some ordering of the basis set $\{[x]|\ x\in K\}$.  

Now $H$ acts a group of automorphisms on $\Pa$. So $\nu_{Q}$ is in fact an $H$-endomorphism of $R^{K}$.By $Irr(H)$ we denote the irreducible $R$-valued characters of $H$.  Given $\chi \in Irr(H)$, we define $N_{\chi}:=\{x \in \R^{K}| g.x=\chi(g)x\ \text{for all}\ g\in H\}$. In other words, $N_{\chi}$ is the direct sum of all irreducible submodules of $\R^{V}$ that affording $\chi$. As $H$ preserves adjacency, $\nu_{Q}$ is an $\R H$ map. Now by applying Schur's Lemma, we have $\nu_{Q}(N_{\chi})\subset N_{\chi}$.

The action of $H$ on $R^{K}$ is the restriction of the natural action of $K^{\times}$ on $R^{K}$. Let $T:\K \to R^{\times}$ be the Teichm\"{u}ller character generating the cyclic group $\mrm{Hom}(\K,\ R^{\times})$. Then $\K$ action on $R^{K}$ decomposes it into the direct sum $R[0]\oplus R^{\K}$. Now the regular module $R^{\K}$ decompose further into a direct sum of $\K$-invariant submodules of rank $1$, affording the characters $T^{i}$, $i=0,\ldots,q-2$. The component affording $T^{i}$ is spanned by $f_{i}:= \sum\limits_{x \in \K} T^{i}(x^{-1})[x]$.     
 Therefore $\{\mathbf{1}, f_{1} \ldots f_{q-2}, [0]\}$ is a basis for $R^K$, where $\mathbf{1}:=f_{0}+[0]=\sum\limits_{x \in K}[x]$. The characters $T^{i}$ and $T^{-i}$ are the same when restricted to $H$. So we have $Irr(H)=\{T^{i}| 0\leq i \leq \ k=\dfrac{q-1}{2}\}$; for $1\leq i<k$, we have $N_{i}:=N_{T^{i}}=Rf_{i}+ R f_{i+k}$; and $N_{0}:=N_{T^{0}}=R[0]+Rf_{k}+R\mathbf{1}$. We now have $R^{K}= \bigoplus\limits_{i=0}^{k-1}N_{i}$ with $\nu_{Q}(N_{i}) \subset N_{i}$ for all $0 \leq i \leq k-1$.
 
Following conventions in \cite{Ax}, we extend the $T^{i}$'s to $K$. As per this convention, the character $T^{0}$ maps every element of $K$ to $1$, while $T^{q-1}$ maps $0$ to $0$. All other characters map $0$ to $0$. For two integers $a,b$ the Jacobi sum $J(T^{a},T^{b})$ is $\sum\limits_{x \in K}T^{a}(x)T^{b}(1-x)$. We refer the reader to Chapter $2$ of \cite{BKR} for formal properties of Jacobi sums. Following the conventions established, for $a \nequiv 0 \pmod{q-1}$, we have $J(T^{a},T^{0})=0$ and $J(T^{a},T^{q-1})=-1$.
 
The following Lemma describes action of $\nu_{Q}$ on $N_{i}$. This result is essentially \cite[Lemma 3.1]{CSX}.
\begin{lemma}\label{im} 
\begin{enumerate}
\item If $1\leq i \leq k-1$, we have 
$\nu_{Q}(f_{i})= \dfrac{1}{2}\left(qf_{i}-J(T^{-i},T^{k})f_{i+k}\right)$.
\item 
$\nu_{Q}(f_{k})=\dfrac{1}{2}\left(-\mathbf{1} +qf_{k} + q[0]\right)$.
\item $\nu_{Q}([0])=\dfrac{1}{2}\left(q[0]- f_{k} -\mathbf{1} \right)$.
\item $\nu_{Q}(\mathbf{1})=0$.
\end{enumerate}
\end{lemma} 
\begin{proof}
We observe that the characteristic function $\delta_{H}$ of $H$ in $K$ is $\dfrac{T^{0}+\psi-\delta_{\{0\}}}{2}$, where $\psi=T^{k}$ is the quadratic character. We now recall that $\nu_{Q}([x])=kx - \sum\limits_{y\in Y} \delta_{H}(y)[x+y]$.

We have
\begin{align*}
2\nu_{Q}(f_{i})&= 2 \sum\limits_{x\in \ K^{\times}}T^{-i}(x)\nu_{Q}([x])\\
&= (q-1)f_{i}- 2 \sum\limits_{x\in \ K^{\times}}T^{-i}(x)\sum\limits_{y \in K} \delta_{H}(y)([x+y]) \\
&= (q-1)f_{i} +  \sum\limits_{x\in \ K^{\times}}T^{-i}(x)\sum\limits_{y \in K} \left(\delta_{0}(y)- T^{0}(y)-\psi(y)\right)[x+y]\\
&= qf_{i}- \sum\limits_{x\in \ K^{\times}}T^{-i}(x)\sum\limits_{y \in K} T^{0}(y)[x+y]- \sum\limits_{x\in \ K^{\times}}T^{-i}(x)\sum\limits_{y \in K} T^{k}(y)[x+y]
\end{align*}

(1) We assume $1\leq i \leq k-1$. In this case,
the middle sum in the above expression $\sum\limits_{x\in \ K^{\times}}T^{-i}(x)\sum\limits_{y \in K} T^{0}(y)[x+y]=\left(\sum\limits_{x\in \ K^{\times}}T^{-i}(x)\right)\times \left(\sum\limits_{z \in K} [z] \right)$. For $i \neq 0$, as $T^{i}$ is a non-trivial character of $K^{\times}$ and thus $\sum\limits_{x\in \ K^{\times}}T^{-i}(x)=0$.

The last sum in the expression above is $$\sum\limits_{x\in \ K^{\times}}T^{-i}(x)\sum\limits_{y \in K} \psi(y)[x+y]=\sum\limits_{x\in \ K^{\times}}T^{-i}(x)\sum\limits_{y \in K} \psi(y)[x+y]= \sum\limits_{z \in K^{\times}}\sum\limits_{x\in \ K^{\times}}T^{-i}(x) \psi(z-x)[z]+\sum\limits_{x\in \ K^{\times}}T^{-i}(x) \psi(-x)[0].$$  For $z\neq 0$, using $T^{-i}(x)\psi(z-x)=T^{-i}(z)\psi(z)T^{-i}(x/z)\psi(1-(x/z))$, we have 
\begin{align*}
\sum\limits_{z \in K^{\times}}\sum\limits_{x\in \ K^{\times}}T^{-i}(x) \psi(z-x)[z] &= \sum\limits_{x,z \in K^{\times}} T^{-i}(x/z)\psi(1-(x/z)) T^{-i}(z)\psi(z)[z] \\
&= \sum\limits_{w,z \in K^{\times}} T^{-i}(w)\psi(w) T^{-i-k}(z)[z]\\
&= J(T^{-i},\psi)f_{i+k}. 
\end{align*}    
We have $\sum\limits_{x\in \ K^{\times}}T^{-i}(x) \psi(-x)[0]= \left(\sum\limits_{x\in \ K^{\times}}T^{-i+k}(x)\right) \psi(-1)[0]$. As $i \neq k$, $T^{-i+k}$ is non trivial and thus $\sum\limits_{x\in \ K^{\times}}T^{-i+k}(x)=0$. We have now proved (i).

The proof of $(2)$ follows by essentially the same computation as above and using the fact that $J(\psi,\psi)=-\psi(-1)=1$. Results (3) and (4) are straightforward.
\end{proof}

\begin{corollary}
The Laplacian $Q$ is similar over $R$ to a diagonal matrix with diagonal entries $J(T^{i},T^{k})$ for $1\leq i \leq q-2$ and $i\neq k$, two ones and one zero. 
\end{corollary}

So computing the $p$-adic valuations of Jacobi sums will give us the $p$-elementary divisors of $Q$.

An integer $a$ not divisible by $q-1$ has, when reduced modulo $q-1$, a unique $p$-digit expansion $a \equiv a_{0}+a_{1}p +\ldots+a_{t-1}p^{t-1} \pmod{q-1}$, where $0 \leq a_{i}\leq p-1$. We represent this expansion by the tuple of digits $(a_{0},\ldots, a_{i}, \ldots,a_{t-1})$.
By $s(a)$ we denote the sum $\sum a_{i}$. For example, $1$ has the expansion $(1, \ldots, 0 ,\ldots 0)$ and $s(1)=1$. 

Applying Stickelberger's theorem on Gauss Sums \cite{Stick} and the well know relation between Gauss and Jacobi sums we can deduce the following theorem.
\begin{theorem}\label{stick}
Let $q$ be a power of a prime $p$ and let $a$ and $b$ be integers not divisible by $q-1$. If $a+b \nequiv 0 \pmod{q-1}$, then we have 
$$v_{p}(J(T^{-a},T^{-b}))= \frac{s(a)+s(b)-s(a+b)}{p-1}.$$  
In other words, the $p$-adic valuation of $J(T^{-a},T^{-b})$ is equal to the number of carries, when adding $p$-expansions of $a$ and $b$ modulo $q-1$.
\end{theorem}
The $p$-adic expansion of $k=\dfrac{q-1}{2}$ is $\sum\limits_{i=0}^{t-1} \dfrac{p-1}{2}p^{i}$ and thus $s(k)=\dfrac{t(p-1)}{2}$. We have $v_{p}(J(T^{-i},T^{k}))=c(i):=\dfrac{s(i)+t(p-1)/2-s(i+k)}{p-1}$. In other words, $c(i)$ is the number of carries when adding the $p$-adic expansions of $i$ and $k$, modulo $q-1$. Observing that $c(i)+c(q-1-i)=t$, we see that $c(i) \leq t$.

We need to solve the following problem in order to find the $p$-elementary divisors of $Q$.

\paragraph{\textbf{The Counting Problem.}} For $1\leq i \leq q-2$ and $i \neq k$, by $c(i)$ we denote the number of carries when adding the $p$-adic expansions of $i$ and $k$, modulo $q-1$. Given $0 \leq a \leq t$, find $e_{a}:=|\{i| \ c(i)=a\}|$. 

The multiplicity of $p^{a}$ as an elementary divisor of $Q$ is $e_{a}$. This problem is the same as the counting problem in \cite[\S 4]{CSX}. The solution found in \cite{CSX} is described in the following Lemma.
\begin{lemma}
\begin{enumerate}
\item $e_{t}= \left(\dfrac{p+1}{2}\right)^{t}-2$;
\item and for $1\leq i <t$, $$e_{i}= \sum\limits_{j=0}^{min\{i, t-i\}} \dfrac{t}{t-j} {t-j \choose j} {t-2j \choose i-j} (-p)^{j} \left(\dfrac{p+1}{2}\right)^{t-2j}.$$
\end{enumerate}
\end{lemma} 
The proof of the above Lemma is in \cite[\S 4]{CSX}. 
The authors used the p-ary add-with-carry algorithm \cite[Theorem 4.1]{X2} and the transfer matrix method \cite[Page 501]{EC1}. Theorem \ref{ptc} immediately follows from this Lemma.    

\section*{Acknowledgement}
I thank Prof. Peter Sin for his valuable suggestions and feedback. 
\bibliographystyle{plain}
\bibliography{my}    

\begin{thebibliography}{10}

\bibitem{Ax}
James Ax.
\newblock Zeroes of polynomials over finite fields.
\newblock {\em American Journal of Mathematics}, 86(2):255--261, 1964.

\bibitem{Bai}
Hua Bai.
\newblock On the critical group of the n-cube.
\newblock {\em Linear algebra and its applications}, 369:251--261, 2003.

\bibitem{BKR}
Bruce~C Berndt, Kenneth~S Williams, and Ronald~J Evans.
\newblock {\em Gauss and {Jacobi} sums}.
\newblock Wiley, 1998.

\bibitem{SDB}
Andries Brouwer, Joshua Ducey, and Peter Sin.
\newblock The elementary divisors of the incidence matrix of skew lines in
  $\mathrm{PG} (3, q)$.
\newblock {\em Proceedings of the American Mathematical Society},
  140(8):2561--2573, 2012.

\bibitem{CSX}
David~B. Chandler, Peter Sin, and Qing Xiang.
\newblock The {Smith} and critical groups of {Paley} graphs.
\newblock {\em Journal of Algebraic Combinatorics}, 41(4):1013--1022, Jun 2015.

\bibitem{DWX}
Cunsheng Ding, Zeying Wang, and Qing Xiang.
\newblock Skew {Hadamard} difference sets from the {Ree–Tits} slice
  symplectic spreads in pg(3,32h+1).
\newblock {\em Journal of Combinatorial Theory, Series A}, 114(5):867 -- 887,
  2007.

\bibitem{DY}
Cunsheng Ding and Jin Yuan.
\newblock A family of skew {Hadamard} difference sets.
\newblock {\em Journal of Combinatorial Theory, Series A}, 113(7):1526 -- 1535,
  2006.

\bibitem{DS}
Joshua~E Ducey, Jonathan Gerhard, and Noah Watson.
\newblock The smith and critical groups of the square rook's graph and its
  complement.
\newblock {\em arXiv preprint arXiv:1507.06583}, 2015.

\bibitem{X2}
Tor Helleseth, Henk~DL Hollmann, Alexander Kholosha, Zeying Wang, and Qing
  Xiang.
\newblock Proofs of two conjectures on ternary weakly regular bent functions.
\newblock {\em IEEE Transactions on Information Theory}, 55(11):5272--5283,
  2009.

\bibitem{JNR}
Brian Jacobson, Andrew Niedermaier, and Victor Reiner.
\newblock Critical groups for complete multipartite graphs and cartesian
  products of complete graphs.
\newblock {\em Journal of Graph Theory}, 44(3):231--250, 2003.

\bibitem{KH}
H.~Kharaghani and B.~Tayfeh-Rezaie.
\newblock A {Hadamard} matrix of order 428.
\newblock {\em Journal of Combinatorial Designs}, 13(6):435--440, 2005.

\bibitem{Lor1}
Dino Lorenzini.
\newblock Smith normal form and {Laplacians}.
\newblock {\em Journal of Combinatorial Theory, Series B}, 98(6):1271 -- 1300,
  2008.

\bibitem{Dlor}
Dino~J Lorenzini.
\newblock A finite group attached to the {Laplacian} of a graph.
\newblock {\em Discrete Mathematics}, 91(3):277--282, 1991.

\bibitem{MW}
T.~S. Michael and W.~D. Wallis.
\newblock Skew-{Hadamard} matrices and the {Smith} normal form.
\newblock {\em Designs, Codes and Cryptography}, 13(2):173--176, Feb 1998.

\bibitem{QM}
Koji Momihara and Qing Xiang.
\newblock Constructions of skew {Hadamard} difference families.
\newblock {\em European Journal of Combinatorics}, 76:73 -- 81, 2019.

\bibitem{MN71}
Morris Newman.
\newblock On the {Smith} normal form.
\newblock {\em J. Res. Nat. Bur. Standards Sect. B}, 75:81--84, 1971.

\bibitem{Pal}
R.~E. A.~C. Paley.
\newblock On orthogonal matrices.
\newblock {\em Journal of Mathematics and Physics}, 12(1-4):311--320, 1933.

\bibitem{P}
Venkata Raghu~Tej Pantangi.
\newblock Critical group of {van Lint-Schrijver} cyclotomic strongly regular
  graphs.
\newblock {\em To appear in Finite fields and their applications}, 2019.

\bibitem{PS}
Venkata Raghu~Tej Pantangi and Peter Sin.
\newblock Smith and critical groups of polar graphs.
\newblock {\em arXiv preprint arXiv:1706.08175}, 2017.

\bibitem{BR}
K.B Reid and Ezra Brown.
\newblock Doubly regular tournaments are equivalent to skew {Hadamard}
  matrices.
\newblock {\em Journal of Combinatorial Theory, Series A}, 12(3):332 -- 338,
  1972.

\bibitem{EC1}
Richard~P. Stanley.
\newblock {\em Enumerative Combinatorics: Volume 1}.
\newblock Cambridge University Press, New York, NY, USA, 2nd edition, 2011.

\bibitem{EC2}
Richard~P. Stanley.
\newblock {\em Enumerative Combinatorics: Volume 2}.
\newblock Cambridge University Press, New York, NY, USA, 2nd edition, 2011.

\bibitem{sta}
Richard~P Stanley.
\newblock Smith normal form in combinatorics.
\newblock {\em Journal of Combinatorial Theory, Series A}, 144:476--495, 2016.

\bibitem{Stick}
L.~Stickelberger.
\newblock Ueber eine verallgemeinerung der {Kreistheilung}.
\newblock {\em Mathematische Annalen}, 37(3):321--367, Sep 1890.

\bibitem{SZ}
G~Szekeres.
\newblock Tournaments and {Hadamard} matrices.
\newblock {\em Enseignement Math}, 15(2):269--278, 1969.

\bibitem{SZ2}
G.~Szekeres.
\newblock Cyclotomy and complementary difference sets.
\newblock {\em Acta Arithmetica}, 18(1):349--353, 1971.

\bibitem{Vince}
A.~Vince.
\newblock Elementary divisors of graphs and matroids.
\newblock {\em European Journal of Combinatorics}, 12(5):445 -- 453, 1991.

\bibitem{WW}
Jennifer Wallis and Albert~Leon Whiteman.
\newblock Some classes of {Hadamard} matrices with constant diagonal.
\newblock {\em Bulletin of the Australian Mathematical Society},
  7(2):233–249, 1972.

\bibitem{W}
Albert~Leon Whiteman.
\newblock An infinite family of skew {Hadamard} matrices.
\newblock {\em Pacific J. Math.}, 38(3):817--822, 1971.

\end{thebibliography}
\end{document}